\documentclass{amsart}
\usepackage[utf8]{inputenc}
\usepackage{amsmath, amssymb, amsthm}
\usepackage{mathrsfs}
\usepackage{geometry}
\newtheorem{theorem}{Theorem}[section]
\newtheorem{lemma}[theorem]{Lemma}

\newtheorem{remark}[theorem]{Remark}

\makeatletter
\@namedef{subjclassname@2020}{\textup{2020} Mathematics Subject Classification}
\makeatother

\begin{document}

\title{Geometric properties of a new hyperbolic type metric}
\author{Xinyu Chen}
\address{School of Science, Zhejiang Sci-Tech University, Hangzhou 310018, China}
\email{ }
\author{Xiaohui Zhang*}
\address{School of Science, Zhejiang Sci-Tech University, Hangzhou 310018, China}
\email{xiaohui.zhang@zstu.edu.cn(*Corresponding author)}

\begin{abstract}
    We introduce a new distance function \(\tilde{S}_{G,c}\) in a metric space \((X,d)\), defined by
    \[
    \tilde{S}_{G,c}(x,y)=\log\left(1+\frac{c\,d(x,y)}{\sqrt{1+d(x,G)}\sqrt{1+d(y,G)}}\right)
    \]
    for \(x, y \in X\), where \(c\) is a positive real number and \(d(x,G)\) denotes the distance from \(x\) to the subset \(G \subset X\). We establish that \(\tilde{S}_{G,c}\) is a metric for \(c \geq 2\). Moreover, we show that the condition \(c \geq 2\) is sharp. This paper investigates geometric properties of the metric \(\tilde{S}_{G,c}\), including comparison inequalities with the triangular ratio metric and inclusion relations among metric balls. We demonstrate the quasiconformality of bilipschitz mappings with respect to \(\tilde{S}_{G,c}\) and study the distortion property of the metric \(\tilde{S}_{\partial\mathbb{B}^{n},c}\) under M\"obius transformations of the unit ball.
\end{abstract}


\newcounter{minutes}\setcounter{minutes}{\time}
\divide\time by 60
\newcounter{hours}\setcounter{hours}{\time}
\multiply\time by 60 \addtocounter{minutes}{-\time}
\def\thefootnote{}
\footnotetext{
\texttt{\tiny File:~\jobname .tex,
          printed: \number\year-\number\month-\number\day,
          \thehours.\ifnum\theminutes<10{0}\fi\theminutes}
}
\makeatletter\def\thefootnote{\@arabic\c@footnote}\makeatother

\maketitle

\section{Introduction}

In geometric function theory, the hyperbolic metric plays a fundamental role and has numerous important applications \cite{BeardonMinda2007, KeenLakic2007}. In higher dimensions, however, the hyperbolic metric is defined only in balls and half-spaces. It is precisely this limitation that motivated the introduction of various hyperbolic-type metrics \cite{GehringOsgood1979, GehringPalka1976}. Examples include the Apollonian metric, the M\"obius-invariant Cassinian metric, Seittenranta's metric, the half-Apollonian metric, the scale-invariant Cassinian metric, and the triangular ratio metric \cite{Beardon1998, HaririKlenVuorinen2020, HaririVuorinenZhang2017, Hasto2006, HastoIbragimovLinden2006, HastoLinden2004, Ibragimov2016, Ibragimov2019, Seittenranta1999, Vuorinen1988, WangVuorinenZhang2021, WangXuVuorinen2021, XuWangZhang2022}. These metrics, often referred to as point-distance metrics, can be classified into one-point and two-point metrics.

Most hyperbolic-type metrics belong to families of relative metrics. A relative metric is defined on a domain \(D \subsetneq \mathbb{R}^{n}\) relative to its boundary. H\"ast\"o \cite{Hasto2002} introduced the generalized relative metric, termed the \(M\)-relative metric, defined on a domain \(D \subsetneq \mathbb{R}^{n}\) by
\[
\rho_{M,D}(x,y) = \sup_{p \in \partial D} \frac{|x-y|}{M(|x-p|, |y-p|)},
\]
where \(M\) is continuous on \((0,\infty) \times (0,\infty)\). An important example is the triangular ratio metric, obtained by choosing \(M(\alpha, \beta) = \alpha + \beta\), i.e.,
\[
t_{D}(x,y) = \sup_{a \in \partial D} \frac{|x-y|}{|x-a| + |y-a|}.
\]
The geometric properties of the triangular ratio metric have been extensively studied \cite{HaririVuorinenZhang2017, HastoLinden2004}.

Let \((X,d)\) be a metric space. In their seminal work, Bonk and Kleiner \cite{BonkKleiner2002} studied the distance function
\[
s_{p}(x,y) = \frac{d(x,y)}{[1 + d(x,p)][1 + d(y,p)]}
\]
for \(x, y \in X\) and a base point \(p \in X\). This function generally fails to satisfy the triangle inequality and thus is not a metric. Through a standard procedure, they defined the metric
\[
\hat{d}_{p}(x,y) := \inf \sum_{i=1}^{k} s_{p}(x_{i}, x_{i-1}),
\]
where the infimum is taken over all finite sequences \(x_{0}, \dots, x_{k} \in X\) with \(x_{0} = x\) and \(x_{k} = y\). They showed that the identity map \(f : (X,d) \to (X,\hat{d}_{p})\) is \(\psi\)-quasim\"obius on an unbounded locally compact metric space \(X\), for any \(p \in X\) and with \(\psi(t) = 16t\). Subsequent research \cite{CuiXiao2022, ZhangXiao2019} has further explored the properties of \(s_p\) and the related metric
\[
S_{p}(x,y) = \log\left(1 + \frac{d(x,y)}{[1 + d(x,p)][1 + d(y,p)]}\right).
\]

Building upon these developments, particularly the work on distance functions like \(s_p\), we introduce a new distance function \(\tilde{s}_{G}\) in a metric space \((X,d)\): for \(G \subset X\) and \(x, y \in X\),
\[
\tilde{s}_{G}(x,y) = \frac{d(x,y)}{\sqrt{1 + d(x)} \sqrt{1 + d(y)}},
\]
where \(d(x) = d(x,G)\) denotes the distance from \(x\) to \(G\). In general, \(\tilde{s}_{G}\) does not satisfy the triangle inequality. For instance, let \(X = \mathbb{R}\) and \(G = \{-4, 4\}\). Then \(\tilde{s}_{G}(3e_{1}, e_{1}) = 1/\sqrt{2}\), \(\tilde{s}_{G}(3e_{1}, 2e_{1}) = 1/\sqrt{6}\), \(\tilde{s}_{G}(e_{1}, 2e_{1}) = 1/(2\sqrt{3})\), and
\[
\tilde{s}_{G}(x,y) > \tilde{s}_{G}(x,z) + \tilde{s}_{G}(y,z).
\]
Thus, \(\tilde{s}_{G}\) is not a metric.

Since \(\tilde{s}_{G}\) itself may not satisfy the triangle inequality, we use it to define a new function
\[
\tilde{S}_{G,c}(x,y) = \log\left(1 + c\tilde{s}_{G}(x,y)\right) = \log\left(1 + \frac{c\,d(x,y)}{\sqrt{1 + d(x)} \sqrt{1 + d(y)}}\right)
\]
for \(x, y \in X\), where \(c > 0\) is a constant. We establish that \(\tilde{S}_{G,c}\) is a metric for \(c \geq 2\), and demonstrate that this condition is generally sharp. This paper investigates key geometric properties of \(\tilde{S}_{G,c}\), including comparison theorems and distortion properties under M\"obius mappings of the unit ball in Euclidean spaces.

The paper is organized as follows. Section 2 proves that \(\tilde{S}_{G,c}\) is a metric for \(c \geq 2\) and provides comparison inequalities with the triangular ratio metric \(t_G\), as well as inclusion relations between the corresponding metric balls. Section 3 studies the convergence of \(\tilde{S}_{G,c}\) under the Hausdorff distance. Section 4 establishes the quasiconformality of bilipschitz mappings with respect to \(\tilde{S}_{G,c}\). Finally, Section 5 analyzes the distortion of \(\tilde{S}_{\partial\mathbb{B}^{n},c}\) under M\"obius transformations of the unit ball \(\mathbb{B}^{n}\).

\section{Comparison Theorems for \(\tilde{S}_{G,c}\)}

We begin by proving that \(\tilde{S}_{G,c}\) is a metric for \(c \geq 2\), and then compare it with the hyperbolic metric and the triangular ratio metric \(t_G\).

\begin{theorem} \label{thm:metric}
    Let \((X,d)\) be a metric space and \(G \subset X\) be a nonempty subset. Then for \(c \geq 2\), the function \(\tilde{S}_{G,c}\) defines a metric on \(X\).
\end{theorem}

\begin{proof}
    Let \(x, y, z \in X\). Clearly, \(\tilde{S}_{G,c}(x,y) \geq 0\), \(\tilde{S}_{G,c}(x,y) = \tilde{S}_{G,c}(y,x)\), and \(\tilde{S}_{G,c}(x,y) = 0\) if and only if \(x = y\). It therefore remains to verify the triangle inequality:
    \[
    \tilde{S}_{G,c}(x,y) \leq \tilde{S}_{G,c}(x,z) + \tilde{S}_{G,c}(y,z) \quad \text{for all } x, y, z \in X.
    \]
    This is equivalent to
    \[
    \log\left(1 + \frac{c\,d(x,y)}{\sqrt{1+d(x)}\sqrt{1+d(y)}}\right) \leq \log\left( \left(1 + \frac{c\,d(x,z)}{\sqrt{1+d(x)}\sqrt{1+d(z)}}\right) \left(1 + \frac{c\,d(y,z)}{\sqrt{1+d(y)}\sqrt{1+d(z)}}\right) \right),
    \]
    which in turn is equivalent to
    \begin{equation} \label{eq:triangle-equiv}
    \frac{d(x,y)}{\sqrt{1+d(x)}\sqrt{1+d(y)}} \leq \frac{d(x,z)}{\sqrt{1+d(x)}\sqrt{1+d(z)}} + \frac{d(y,z)}{\sqrt{1+d(y)}\sqrt{1+d(z)}} + \frac{c\,d(x,z)d(y,z)}{(1+d(z))\sqrt{1+d(x)}\sqrt{1+d(y)}}.
    \end{equation}
    Without loss of generality, assume \(d(x) \leq d(y)\). We consider three cases.

    \textbf{Case 1: \(d(z) \leq d(x) \leq d(y)\).} By the triangle inequality for \(d\),
    \begin{align*}
    \frac{d(x,y)}{\sqrt{1+d(x)}\sqrt{1+d(y)}} &\leq \frac{d(x,z)}{\sqrt{1+d(x)}\sqrt{1+d(y)}} + \frac{d(y,z)}{\sqrt{1+d(x)}\sqrt{1+d(y)}} \\
    &\leq \frac{d(x,z)}{\sqrt{1+d(x)}\sqrt{1+d(z)}} + \frac{d(y,z)}{\sqrt{1+d(y)}\sqrt{1+d(z)}},
    \end{align*}
    which establishes \eqref{eq:triangle-equiv}.

    \textbf{Case 2: \(d(x) \leq d(z) \leq d(y)\).} Since \(1 + d(x) \leq \sqrt{1+d(x)}\sqrt{1+d(z)}\), the triangle inequality gives
    \[
    1 + d(z) \leq 1 + d(x) + d(x,z) \leq \sqrt{1+d(x)}\sqrt{1+d(z)} + d(x,z),
    \]
    which implies
    \[
    \frac{1}{\sqrt{1+d(x)}} \leq \frac{1}{\sqrt{1+d(z)}} + \frac{d(x,z)}{(1+d(z))\sqrt{1+d(x)}}.
    \]
    Multiplying both sides by \(d(y,z)/\sqrt{1+d(y)}\) yields
    \begin{equation} \label{eq:case2-1}
    \frac{d(y,z)}{\sqrt{1+d(x)}\sqrt{1+d(y)}} \leq \frac{d(y,z)}{\sqrt{1+d(z)}\sqrt{1+d(y)}} + \frac{d(x,z)d(y,z)}{(1+d(z))\sqrt{1+d(x)}\sqrt{1+d(y)}}.
    \end{equation}
    Now, applying the triangle inequality \(d(x,y) \leq d(x,z) + d(y,z)\) gives
    \begin{equation} \label{eq:case2-2}
    \frac{d(x,y)}{\sqrt{1+d(x)}\sqrt{1+d(y)}} \leq \frac{d(x,z)}{\sqrt{1+d(x)}\sqrt{1+d(y)}} + \frac{d(y,z)}{\sqrt{1+d(x)}\sqrt{1+d(y)}}.
    \end{equation}
    Furthermore, since \(d(z) \leq d(y)\),
    \begin{equation} \label{eq:case2-3}
    \frac{d(x,z)}{\sqrt{1+d(x)}\sqrt{1+d(y)}} \leq \frac{d(x,z)}{\sqrt{1+d(x)}\sqrt{1+d(z)}}.
    \end{equation}
    Combining inequalities \eqref{eq:case2-1}, \eqref{eq:case2-2}, and \eqref{eq:case2-3}, we obtain
    \[
    \frac{d(x,y)}{\sqrt{1+d(x)}\sqrt{1+d(y)}} \leq \frac{d(x,z)}{\sqrt{1+d(x)}\sqrt{1+d(z)}} + \frac{d(y,z)}{\sqrt{1+d(y)}\sqrt{1+d(z)}} + \frac{d(x,z)d(y,z)}{(1+d(z))\sqrt{1+d(x)}\sqrt{1+d(y)}},
    \]
    which implies \eqref{eq:triangle-equiv} for \(c \geq 1\).

    \textbf{Case 3: \(d(x) \leq d(y) \leq d(z)\).} Since \(d(x) \leq d(z)\) and \(d(y) \leq d(z)\), arguments similar to those in Case 2 yield
    \[
    \frac{d(y,z)}{\sqrt{1+d(x)}\sqrt{1+d(y)}} \leq \frac{d(y,z)}{\sqrt{1+d(z)}\sqrt{1+d(y)}} + \frac{d(x,z)d(y,z)}{(1+d(z))\sqrt{1+d(x)}\sqrt{1+d(y)}}
    \]
    and
    \[
    \frac{d(x,z)}{\sqrt{1+d(x)}\sqrt{1+d(y)}} \leq \frac{d(x,z)}{\sqrt{1+d(z)}\sqrt{1+d(x)}} + \frac{d(x,z)d(y,z)}{(1+d(z))\sqrt{1+d(x)}\sqrt{1+d(y)}}.
    \]
    Therefore,
   \begin{align*}
    \frac{d(x,y)}{\sqrt{1+d(x)}\sqrt{1+d(y)}} &\leq \frac{d(x,z)}{\sqrt{1+d(x)}\sqrt{1+d(y)}} + \frac{d(y,z)}{\sqrt{1+d(x)}\sqrt{1+d(y)}}\\
    &\leq \frac{d(x,z)}{\sqrt{1+d(z)}\sqrt{1+d(x)}} + \frac{d(y,z)}{\sqrt{1+d(z)}\sqrt{1+d(y)}} + \frac{2d(x,z)d(y,z)}{(1+d(z))\sqrt{1+d(x)}\sqrt{1+d(y)}},
    \end{align*}
    which implies \eqref{eq:triangle-equiv} for \(c \geq 2\). This completes the proof.
\end{proof}

\begin{remark} \label{rem:sharpness}
    The condition \(c \geq 2\) in Theorem \ref{thm:metric} is sharp. To see this, consider \(X = \mathbb{R}\) and \(G = \{-M, M\}\) for \(M > 0\). Define \(A(x,y) = d(x,y)/(\sqrt{1+d(x)}\sqrt{1+d(y)})\). If \(\tilde{S}_{G,c}\) is a metric, then
    \[
    A(x,y) \leq A(x,z) + A(y,z) + c A(x,z)A(y,z),
    \]
    which is equivalent to
    \[
    c \geq \frac{A(x,y) - A(x,z) - A(y,z)}{A(x,z)A(y,z)} \equiv h(x,y,z)
    \]
    for all distinct \(x, y, z \in \mathbb{R}\). Setting \(x = -y = t\) for \(t \in (0, M]\) and \(z = 0\), a straightforward computation gives
    \[
    h(x,y,z) = \frac{2\sqrt{1+M}}{\sqrt{1+M} + \sqrt{1+M-t}}.
    \]
    Taking \(t = M\), we find \(c \geq 2\sqrt{1+M}/(\sqrt{1+M} + 1) \to 2\) as \(M \to +\infty\).
\end{remark}

Next, we compare the metric \(\tilde{S}_{\partial\mathbb{B}^{n},c}\) with the hyperbolic metric \(\rho_{\mathbb{B}^{n}}\) in the unit ball \(\mathbb{B}^{n}\) and with the triangular ratio metric \(s_D\) for a proper subdomain \(D\) of \(\mathbb{R}^{n}\). Recall that the hyperbolic metric \(\rho_{\mathbb{B}^{n}}\) satisfies \cite{Beardon1995}
\[
\tanh\frac{\rho_{\mathbb{B}^{n}}(x,y)}{2} = \frac{|x-y|}{\sqrt{|x-y|^{2} + (1-|x|^{2})(1-|y|^{2})}}, \quad \text{for all } x, y \in \mathbb{B}^{n} \setminus \{0\}.
\]

\begin{theorem} \label{thm:comparison-hyperbolic}
    For all \(x, y \in \mathbb{B}^{n} \setminus \{0\}\),
    \[
    \tilde{S}_{\partial\mathbb{B}^{n},c}(x,y) \leq \log\left(1 + 2c \cdot \tanh\frac{\rho_{\mathbb{B}^{n}}(x,y)}{2} \right).
    \]
\end{theorem}

\begin{proof}
    For \(x, y \in \mathbb{B}^{n}\), we have
    \[
    \tilde{S}_{\partial\mathbb{B}^{n},c}(x,y) = \log\left(1 + \frac{c|x-y|}{\sqrt{2-|x|}\sqrt{2-|y|}} \right)
    \]
    and
    \[
    \frac{e^{\tilde{S}_{\partial\mathbb{B}^{n},c}(x,y)} - 1}{\tanh(\rho_{\mathbb{B}^{n}}(x,y)/2)} = c \sqrt{ \frac{|x-y|^{2} + (1-|x|^{2})(1-|y|^{2})}{(2-|x|)(2-|y|)} }.
    \]
    Using the inequality \(|x-y| \leq |x| + |y|\), we obtain
    \[
    \frac{e^{\tilde{S}_{\partial\mathbb{B}^{n},c}(x,y)} - 1}{\tanh(\rho_{\mathbb{B}^{n}}(x,y)/2)} \leq c \sqrt{ \frac{(|x|+|y|)^{2} + (1-|x|^{2})(1-|y|^{2})}{(2-|x|)(2-|y|)} } = \frac{c(1 + |x||y|)}{\sqrt{(2-|x|)(2-|y|)}} \leq 2c,
    \]
    from which the desired inequality follows.
\end{proof}

\begin{theorem} \label{thm:comparison-triangular}
    Let \(G \subsetneq \mathbb{R}^{n}\) be a domain. For all \(x, y \in G\),
    \[
    \log\left(1 + \frac{2c\,d_{xy}\,t_{G}(x,y)}{1 + d_{xy}} \right) \leq \tilde{S}_{\partial G,c}(x,y) \leq \log\left(1 + \frac{2c\,d_{xy}\,t_{G}(x,y)}{(1 + d_{xy})(1 - t_{G}(x,y))} \right),
    \]
    where \(d_{xy} = \min\{d(x), d(y)\}\) and \(c \geq 2\).
\end{theorem}

\begin{proof}
    For \(x, y \in G\), let \(p_{0} \in \partial G\) such that \(d(x) = |x - p_{0}|\). By the triangle inequality, for any \(p \in \partial G\),
    \[
    |x-p| + |p-y| \leq |x-p_{0}| + |y-p_{0}| \leq 2|x-p_{0}| + |x-y| = 2d(x) + |x-y|,
    \]
    so that
    \[
    \inf_{p \in \partial G} \{ |x-p| + |p-y| \} \leq 2d(x) + |x-y|.
    \]
    Similarly,
    \[
    \inf_{p \in \partial G} \{ |x-p| + |p-y| \} \leq 2d(y) + |x-y|.
    \]
    Hence,
    \[
    \inf_{p \in \partial G} \{ |x-p| + |p-y| \} \leq 2d_{xy} + |x-y|.
    \]
    By the definition of the triangular ratio metric,
    \begin{align}
    t_{G}(x,y)& = \sup_{p \in \partial G} \frac{|x-y|}{|x-p| + |p-y|} = \frac{|x-y|}{\inf_{p} \{ |x-p| + |p-y| \}}\nonumber\\
    & \geq \frac{|x-y|}{2d_{xy} + |x-y|} = \frac{|x-y|/(1+d_{xy})}{2d_{xy}/(1+d_{xy}) + |x-y|/(1+d_{xy})}.\label{eq:tG-lower-bound}
    \end{align}
    From the definition of \(\tilde{S}_{\partial G,c}\),
    \[
    e^{\tilde{S}_{\partial G,c}(x,y)} - 1 = \frac{c|x-y|}{\sqrt{1+d(x)}\sqrt{1+d(y)}} \leq \frac{c|x-y|}{1 + d_{xy}}.
    \]
    Substituting into \eqref{eq:tG-lower-bound} yields
    \[
    t_{G}(x,y) \geq \frac{e^{\tilde{S}_{\partial G,c}(x,y)} - 1}{2c\,d_{xy}/(1+d_{xy}) + e^{\tilde{S}_{\partial G,c}(x,y)} - 1},
    \]
    which is equivalent to
    \[
    e^{\tilde{S}_{\partial G,c}(x,y)} - 1 \leq \frac{2c\,d_{xy}\,t_{G}(x,y)}{(1+d_{xy})(1 - t_{G}(x,y))},
    \]
    establishing the upper bound.

    For the lower bound, observe that
    \[
    t_{G}(x,y) \leq \frac{|x-y|}{d(x) + d(y)}
    \]
    and
    \begin{align*}
    e^{\tilde{S}_{\partial G,c}(x,y)} - 1 &= \frac{c|x-y|}{\sqrt{1+d(x)}\sqrt{1+d(y)}} \geq \frac{2c|x-y|}{2 + d(x) + d(y)} \\
    &= \frac{2c|x-y|/(d(x)+d(y))}{1 + 2/(d(x)+d(y))} \geq \frac{2c|x-y|/(d(x)+d(y))}{1 + 1/d_{xy}} \\
    &\geq \frac{2c\,t_{G}(x,y)}{1 + 1/d_{xy}} = \frac{2c\,d_{xy}\,t_{G}(x,y)}{1 + d_{xy}}.
    \end{align*}
    This implies the lower bound and completes the proof.
\end{proof}

Let \((X,d)\) be a metric space. A metric ball \(B_{d}(x,r)\) is defined as
\[
B_{d}(x,r) = \{ y \in X : d(x,y) < r \}.
\]
The following theorem describes inclusion relations between balls in the triangular ratio metric \(t_D\) and the metric \(\tilde{S}_{\partial G,c}\).

\begin{theorem} \label{thm:ball-inclusion}
    Let \(D \subsetneq \mathbb{R}^{n}\) be a domain, \(x \in D\), and \(r > 0\). Then
    \[
    B_{t}(x, \ell) \subset B_{\tilde{S}}(x, r) \subset B_{t}(x, L),
    \]
    where
    \[
    \ell = \frac{(e^{r} - 1)(1 + d(x))}{(e^{r} - 1)(1 + d(x)) + 2c\,d(x)}, \quad L = \frac{(e^{r} - 1)(1 + d(x))}{c\,d(x)}.
    \]
\end{theorem}

\begin{proof}
    If \(y \in B_{t}(x, \ell)\), then by \eqref{eq:tG-lower-bound},
    \[
    \frac{|x-y|/(1+d_{xy})}{2d_{xy}/(1+d_{xy}) + |x-y|/(1+d_{xy})} \leq t_{D}(x,y) < \ell \leq \frac{(e^{r} - 1)(1 + d_{xy})}{(e^{r} - 1)(1 + d_{xy}) + 2c\,d_{xy}},
    \]
    which implies
    \[
    \frac{|x-y|}{1 + d_{xy}} < \frac{e^{r} - 1}{c}.
    \]
    Therefore,
    \[
    e^{\tilde{S}_{\partial D,c}(x,y)} - 1 = \frac{c|x-y|}{\sqrt{1+d(x)}\sqrt{1+d(y)}} \leq \frac{c|x-y|}{1 + d_{xy}} < e^{r} - 1,
    \]
    so \(\tilde{S}_{\partial D,c}(x,y) < r\) and \(y \in B_{\tilde{S}}(x,r)\). This proves \(B_{t}(x, \ell) \subset B_{\tilde{S}}(x, r)\).

    Now suppose \(y \in B_{\tilde{S}}(x, r)\). Then, from the inequality
    \[
    \frac{2c|x-y|}{2 + d(x) + d(y)} \leq e^{\tilde{S}_{\partial G,c}(x,y)} - 1 < e^{r} - 1,
    \]
    we obtain
    \[
    \frac{|x-y|}{d(x) + d(y)} < \frac{(e^{r} - 1)(1 + 2/(d(x)+d(y)))}{2c} \leq \frac{(e^{r} - 1)(1 + d(x))}{c\,d(x)} = L.
    \]
    Since \(t_{D}(x,y) \leq |x-y|/(d(x) + d(y))\), it follows that \(t_{D}(x,y) < L\), hence \(B_{\tilde{S}}(x, r) \subset B_{t}(x, L)\).
\end{proof}

\section{Convergence of \(\tilde{S}_{G,c}\) under Hausdorff Distance}

This section studies the convergence behavior of the metric \(\tilde{S}_{G,c}\) when the subset \(G\) varies with respect to the Hausdorff distance. Let \((X,d)\) be a metric space. For \(z \in X\) and \(A \subset X\), the distance from \(z\) to \(A\) is defined as \(d(z,A) = \inf \{ d(z,x) : x \in A \}\). The Hausdorff distance between subsets \(A, B \subset X\) is given by
\[
d_H(A,B) = \max\left\{ \sup_{b \in B} d(b,A),\, \sup_{a \in A} d(a,B) \right\} = \inf \left\{ \varepsilon > 0 : A \subset B(\varepsilon),\, B \subset A(\varepsilon) \right\},
\]
where \(A(\varepsilon) = \{ y \in X : d(y,A) < \varepsilon \}\).

We begin with a key estimate relating the distance functions to the Hausdorff metric.

\begin{lemma} \label{lem:Hausdorff-distance}
Let \(G_1\) and \(G_2\) be two nonempty bounded closed subsets of a metric space \((X,d)\). Then
\[
\left| d(x,G_2) - d(x,G_1) \right| \leq d_H(G_1,G_2)
\]
for all \(x \in X\).
\end{lemma}

\begin{proof}
Without loss of generality, assume \(d(x,G_2) \geq d(x,G_1)\). For any \(\varepsilon > 0\), there exists \(x_1 \in G_1\) such that
\[
d(x,x_1) - \varepsilon \leq d(x,G_1) \leq d(x,x_1).
\]
Then for any \(x_2 \in G_2\),
\[
d(x,G_2) - d(x,G_1) \leq d(x,G_2) - d(x,x_1) + \varepsilon \leq d(x,x_2) - d(x,x_1) + \varepsilon \leq d(x_1,x_2) + \varepsilon.
\]
Taking the infimum over \(x_2 \in G_2\) yields
\[
d(x,G_2) - d(x,G_1) \leq d(x_1,G_2) + \varepsilon \leq d_H(G_1,G_2) + \varepsilon.
\]
Since \(\varepsilon > 0\) is arbitrary, the desired inequality follows.
\end{proof}

Using this lemma, we establish the continuity of \(\tilde{S}_{G,c}\) with respect to Hausdorff convergence.

\begin{theorem} \label{thm:Hausdorff-convergence}
Let \(\{G_n\}_{n=1}^{\infty}\) be a sequence of nonempty bounded closed subsets of a metric space \((X,d)\). If \(d_H(G_n, G) \to 0\) as \(n \to \infty\), then
\[
\lim_{n \to \infty} \tilde{S}_{G_n,c}(x,y) = \tilde{S}_{G,c}(x,y)
\]
for all \(x, y \in X\).
\end{theorem}

\begin{proof}
Let \(\varepsilon_n = d_H(G_n, G)\). By Lemma \ref{lem:Hausdorff-distance}, for all \(x \in X\),
\[
d(x,G) - \varepsilon_n \leq d(x,G_n) \leq d(x,G) + \varepsilon_n.
\]
Since \(\varepsilon_n \to 0\) as \(n \to \infty\), it follows that \(\lim_{n \to \infty} d(x,G_n) = d(x,G)\). Therefore,
\begin{align*}
\lim_{n \to \infty} \tilde{S}_{G_n,c}(x,y) &= \lim_{n \to \infty} \log\left(1 + \frac{c\,d(x,y)}{\sqrt{1 + d(x,G_n)} \sqrt{1 + d(y,G_n)}} \right) \\
&= \log\left(1 + \frac{c\,d(x,y)}{\sqrt{1 + \lim_{n \to \infty} d(x,G_n)} \sqrt{1 + \lim_{n \to \infty} d(y,G_n)}} \right) \\
&= \log\left(1 + \frac{c\,d(x,y)}{\sqrt{1 + d(x,G)} \sqrt{1 + d(y,G)}} \right) = \tilde{S}_{G,c}(x,y),
\end{align*}
which completes the proof.
\end{proof}

\section{Quasiconformality of Bilipschitz Mappings in the Metric \(\tilde{S}_{G,c}\)}

We now investigate the quasiconformality of bilipschitz mappings with respect to the metric \(\tilde{S}_{G,c}\). Let us first recall the relevant definitions.

Let \((X,d_X)\) and \((Y,d_Y)\) be metric spaces, and let \(L \geq 1\). A mapping \(f : X \to Y\) is called \emph{\(L\)-Lipschitz} if
\[
d_Y(f(x),f(y)) \leq L\,d_X(x,y) \quad \text{for all } x, y \in X.
\]
It is called \emph{\(L\)-bilipschitz} if
\[
\frac{1}{L}d_X(x,y) \leq d_Y(f(x),f(y)) \leq L\,d_X(x,y) \quad \text{for all } x, y \in X.
\]

For a homeomorphism \(f : X \to Y\) between metric spaces, define for \(x \in X\) and \(r > 0\):
\[
L_f(x,r) = \sup \{ d_Y(f(x),f(y)) : d_X(x,y) \leq r \}, \quad
l_f(x,r) = \inf \{ d_Y(f(x),f(y)) : d_X(x,y) \geq r \}.
\]
The \emph{linear dilatation} of \(f\) at \(x\) is defined as
\[
H_f(x) = \limsup_{r \to 0} \frac{L_f(x,r)}{l_f(x,r)}.
\]
The mapping \(f\) is called \emph{quasiconformal} if there exists a constant \(H < \infty\) such that \(H_f(x) \leq H\) for all \(x \in X\). For comprehensive treatments of quasiconformal mappings, we refer to \cite{GehringMartinPalka2017, Vaisala1971}.

The following theorem provides a sharp estimate for the linear dilatation of a bilipschitz mapping with respect to the metrics \(\tilde{S}_{G,c}\).

\begin{theorem} \label{thm:quasiconformality}
Let \(G\) be a nonempty proper subdomain of a proper metric space \((X,d_X)\), and let \(f : (G,d_X) \to (f(G),d_Y) \subset (Y,d_Y)\) be a sense-preserving homeomorphism that is \(L\)-bilipschitz with respect to the metrics \(\tilde{S}_{\partial G,c}\) and \(\tilde{S}_{\partial f(G),c}\), i.e.,
\[
\frac{1}{L} \tilde{S}_{\partial G,c}(x,y) \leq \tilde{S}_{\partial f(G),c}(f(x),f(y)) \leq L\, \tilde{S}_{\partial G,c}(x,y)
\]
for all \(x, y \in G\). Then \(f\) is quasiconformal with respect to the original metrics \(d_X\) and \(d_Y\), with linear dilatation satisfying \(H(f) \leq L^2\).
\end{theorem}

\begin{proof}
We first establish two-sided estimates for the metric \(\tilde{S}_{\partial D,c}\). Let \((Z,d)\) be a metric space and \(D \subset Z\). For \(x, y \in D\),
\begin{align*}
\tilde{S}_{\partial D,c}(x,y) &= \log\left(1 + \frac{c\,d(x,y)}{\sqrt{1+d(x)}\sqrt{1+d(y)}} \right) \\
&\leq \log\left(1 + \frac{c\,d(x,y)}{\min\{1+d(x), 1+d(y)\}} \right),
\end{align*}
and
\[
\tilde{S}_{\partial D,c}(x,y) \geq \log\left(1 + \frac{2c\,d(x,y)}{2 + d(x) + d(y)} \right).
\]
These inequalities imply the following bounds:
\begin{equation} \label{eq:metric-bounds}
\frac{\min\{1+d(x), 1+d(y)\}}{c} \left(e^{\tilde{S}_{\partial D,c}} - 1\right) \leq d(x,y) \leq \frac{2 + d(x) + d(y)}{2c} \left(e^{\tilde{S}_{\partial D,c}} - 1\right).
\end{equation}

Now, fix \(x \in G\) and let \(r > 0\). Since the metric space $(X,d_X)$ is proper and $f$ is a homeomorphism, we can choose points \(z, w \in G\) with \(d_X(x,z) = d_X(x,w) = r\) such that
\[
d_Y(f(x),f(z)) = \sup \{ d_Y(f(x),f(y)) : d_X(x,y) \leq r \},
\]
\[
d_Y(f(x),f(w)) = \inf \{ d_Y(f(x),f(y)) : d_X(x,y) \geq r \}.
\]
Define
\[
F(x,z,w) = \frac{2 + d_Y(f(x)) + d_Y(f(z))}{2 \min\{1 + d_Y(f(x)), 1 + d_Y(f(w))\}},
\]
and observe that \(F(x,z,w) \to 1\) as \(r \to 0\).

Applying the estimates \eqref{eq:metric-bounds} and using the asymptotic relations \((1+t)^\alpha - 1 \sim \alpha t\) and \(\log(1+u) \sim u\) as \(t, u \to 0\), we obtain
\begin{align*}
\frac{d_Y(f(x),f(z))}{d_Y(f(x),f(w))}
&\leq F(x,z,w) \frac{e^{\tilde{S}_{\partial f(G),c}(f(x),f(z))} - 1}{e^{\tilde{S}_{\partial f(G),c}(f(x),f(w))} - 1} \\
&\leq F(x,z,w) \frac{e^{L \tilde{S}_{\partial G,c}(x,z)} - 1}{e^{\tilde{S}_{\partial G,c}(x,w)/L} - 1} \\
&\leq F(x,z,w) \frac{\left(1 + \frac{c\,d_X(x,z)}{\min\{1+d_X(x),1+d_X(z)\}} \right)^L - 1}{\left(1 + \frac{2c\,d_X(x,w)}{2 + d_X(x) + d_X(w)} \right)^{1/L} - 1} \to L^2
\end{align*}
as \(r = d_X(x,z) = d_X(x,w) \to 0\). Therefore,
\[
H(f) = \limsup_{r \to 0} \frac{d_Y(f(x),f(z))}{d_Y(f(x),f(w))} \leq L^2,
\]
which completes the proof.
\end{proof}

\section{Distortion Properties under M\"obius Transformations}

In this final section, we study the distortion of the metric \(\tilde{S}_{G,c}\) under M\"obius transformations. We work in the \(n\)-dimensional Euclidean space \(\mathbb{R}^n\) with \(n \geq 2\). Denote by \(B^n(x,r) = \{ y \in \mathbb{R}^n : |x-y| < r \}\) the open ball and by \(S^{n-1}(a,r) = \{ x \in \mathbb{R}^n : |x-a| = r \}\) the sphere. For brevity, write \(\mathbb{B}^n = B^n(0,1)\) and \(\mathbb{S}^{n-1} = S^{n-1}(0,1)\).

The reflection in \(S^{n-1}(a,r)\) is the map \(\psi\) defined by
\[
\psi(x) = a + \frac{r^2(x-a)}{|x-a|^2} \quad \text{for } x \in \mathbb{R}^n \setminus \{a\},
\]
with \(\psi(a) = \infty\) and \(\psi(\infty) = a\). The reflection in \(\mathbb{S}^{n-1}\) is denoted by \(x \mapsto x^*\), where \(x^* = x/|x|^2\) for \(x \neq 0, \infty\). For any reflection \(\psi\) in \(S^{n-1}(a,r)\), we have the formula \cite{Beardon1995}
\[
|\psi(x) - \psi(y)| = \frac{r^2 |x-y|}{|x-a|\,|y-a|}, \quad x, y \in \mathbb{R}^n \setminus \{a\}.
\]

We analyze the distortion of \(\tilde{S}_{\mathbb{S}^{n-1},c}\) under M\"obius transformations of the unit ball. By \cite[Theorem 3.5.1]{Beardon1995}, if \(\phi\) is a M\"obius transformation with \(\phi(\mathbb{B}^n) = \mathbb{B}^n\), then \(\phi(x) = A(\sigma(x))\), where \(\sigma\) is a reflection in a sphere orthogonal to \(\mathbb{S}^{n-1}\) and \(A\) is an orthogonal matrix.

We will use the following lemmas.

\begin{lemma}[\cite{SimicVuorinenWang2015}] \label{lem:mobius-distortion}
Let \(a, b \in \mathbb{B}^n\). Then
\[
\frac{||b| - |a||}{1 - |a||b|} \leq \frac{|b - a|}{|a||b - a^*|} \leq \frac{|b| + |a|}{1 + |a||b|}.
\]
\end{lemma}

\begin{lemma} \label{lem:log-inequality}
For \(x, y > 0\) with \(x \geq y\),
\[
\frac{\log(1+x)}{\log(1+y)} \leq \frac{x}{y}.
\]
\end{lemma}

\begin{proof}
Consider the function \(f(t) = \log(1+t)/t\) for \(t > 0\). Since \(f'(t) < 0\), \(f\) is strictly decreasing. Therefore, for \(x \geq y > 0\), we have \(f(x) \leq f(y)\), which is equivalent to the desired inequality.
\end{proof}

We now state and prove the main distortion theorem.

\begin{theorem} \label{thm:mobius-distortion}
Let \(a \in \mathbb{B}^n\) and \(\phi\) be a M\"obius transformation with \(\phi(\mathbb{B}^n) = \mathbb{B}^n\) and \(\phi(a) = 0\). Then for all \(x, y \in \mathbb{B}^n\),
\[
\frac{1 - |a|}{1 + |a|} \, \tilde{S}_{\mathbb{S}^{n-1},c}(x,y) \leq \tilde{S}_{\mathbb{S}^{n-1},c}(\phi(x),\phi(y)) \leq \frac{1 + |a|}{1 - |a|} \, \tilde{S}_{\mathbb{S}^{n-1},c}(x,y).
\]
\end{theorem}

\begin{proof}
The inequality is trivial if \(x = y\), so assume \(x \neq y\). If \(a = 0\), then \(\phi\) is an orthogonal transformation and the result is immediate. Now suppose \(a \neq 0\).

Since \(\tilde{S}_{\mathbb{S}^{n-1},c}\) is invariant under orthogonal transformations and \(\phi(x) = A(\sigma(x))\), where \(\sigma\) is a reflection in a sphere \(S^{n-1}(a^*, r)\) orthogonal to \(\mathbb{S}^{n-1}\) with \(r = \sqrt{|a^*|^2 - 1}\) and \(\sigma(a) = 0\), we have
\[
\tilde{S}_{\mathbb{S}^{n-1},c}(\phi(x),\phi(y)) = \tilde{S}_{\mathbb{S}^{n-1},c}(\sigma(x),\sigma(y)).
\]
For \(x, y \in \mathbb{B}^n\),
\[
|\sigma(x) - \sigma(y)| = \frac{r^2 |x-y|}{|x-a^*|\,|y-a^*|}, \quad |\sigma(x)| = |\sigma(x) - \sigma(a)| = \frac{|x-a|}{|a|\,|x-a^*|}.
\]
By Lemma \ref{lem:mobius-distortion},
\[
\frac{|x-a|}{|a|\,|x-a^*|} \leq \frac{|x| + |a|}{1 + |x||a|},
\]
which implies
\begin{equation} \label{eq:reflection-bound}
|x - a| \leq \frac{|a|(|x| + |a|)}{1 + |x||a|} |x - a^*|.
\end{equation}

Let us define
\[
A(|x-y|) = \frac{c(1 - |a|^2)|x-y|}{|a|}, \quad
B(t) = \sqrt{(2|a|^2 - |a|)t + 2|a| - |a|^2}.
\]
Then
\begin{align*}
&\frac{\tilde{S}_{\mathbb{S}^{n-1},c}(\sigma(x),\sigma(y))}{\tilde{S}_{\mathbb{S}^{n-1},c}(x,y)}
= \frac{\log\left(1 + \dfrac{c|\sigma(x)-\sigma(y)|}{\sqrt{2-|\sigma(x)|}\sqrt{2-|\sigma(y)|}} \right)}{\log\left(1 + \dfrac{c|x-y|}{\sqrt{2-|x|}\sqrt{2-|y|}} \right)} \\
&= \frac{\log\left(1 + \dfrac{A(|x-y|)}{\sqrt{|x-a^*|}\,\sqrt{|y-a^*|} \sqrt{2|a|\,|x-a^*| - |x-a|} \sqrt{2|a|\,|y-a^*| - |y-a|}} \right)}{\log\left(1 + \dfrac{c|x-y|}{\sqrt{2-|x|}\sqrt{2-|y|}} \right)}.
\end{align*}
Using inequality \eqref{eq:reflection-bound}, we obtain
\[
2|a|\,|x-a^*| - |x-a| \geq  2|a|\,|x-a^*|-\frac{|a|(|x|+|a|)}{1 + |a||x|}\,|x-a^*| = \frac{(2|a|^2-|a|)|x|+2|a|-|a|^2}{1 + |a||x|} |x-a^*|.
\]
Thus,
\[
\sqrt{2|a|\,|x-a^*| - |x-a|} \geq \frac{B(|x|)\sqrt{|x-a^*|}}{ \sqrt{1 + |a||x|} }.
\]
A similar bound holds for the term involving \(y\). Therefore,
\begin{align*}
\frac{\tilde{S}_{\mathbb{S}^{n-1},c}(\sigma(x),\sigma(y))}{\tilde{S}_{\mathbb{S}^{n-1},c}(x,y)}
&\leq \frac{\log\left(1 + \dfrac{A(|x-y|) \sqrt{1+|a||x|} \sqrt{1+|a||y|}}{B(|x|)B(|y|) |x-a^*|\,|y-a^*|} \right)}{\log\left(1 + \dfrac{c|x-y|}{\sqrt{2-|x|}\sqrt{2-|y|}} \right)} \\
&\leq \frac{\log\left(1 + \dfrac{A(|x-y|) \sqrt{1+|a||x|} \sqrt{1+|a||y|}}{B(|x|)B(|y|) (|a^*| - |x|)(|a^*| - |y|)} \right)}{\log\left(1 + \dfrac{c|x-y|}{\sqrt{2-|x|}\sqrt{2-|y|}} \right)},
\end{align*}
where the last inequality follows from \(|x - a^*| \geq |a^*| - |x|\).

Define the function
\[
G(s,t) = \frac{ \sqrt{1-|a|^2} \sqrt{1+|a|s} \sqrt{2-s} }{ (1-|a|s) \sqrt{(2|a|-1)s + 2 - |a|} } \cdot \frac{ \sqrt{1-|a|^2} \sqrt{1+|a|t} \sqrt{2-t} }{ (1-|a|t) \sqrt{(2|a|-1)t + 2 - |a|} }.
\]
Then the ratio of the arguments inside the logarithms is precisely \(G(|x|,|y|)\).

If \(G(|x|,|y|) \leq 1\), then clearly
\[
\frac{\tilde{S}_{\mathbb{S}^{n-1},c}(\sigma(x),\sigma(y))}{\tilde{S}_{\mathbb{S}^{n-1},c}(x,y)} \leq 1 \leq \frac{1+|a|}{1-|a|}.
\]
If \(G(|x|,|y|) > 1\), then by Lemma \ref{lem:log-inequality},
\[
\frac{\tilde{S}_{\mathbb{S}^{n-1},c}(\sigma(x),\sigma(y))}{\tilde{S}_{\mathbb{S}^{n-1},c}(x,y)} \leq G(|x|,|y|).
\]
Thus, it suffices to prove that
\[
\frac{ \sqrt{1-|a|^2} \sqrt{1+|a|s} \sqrt{2-s} }{ (1-|a|s) \sqrt{(2|a|-1)s + 2 - |a|} } \leq \sqrt{ \frac{1+|a|}{1-|a|} } \quad \text{for } s \in [0,1].
\]
Squaring both sides, this inequality is equivalent to
\[
(1-|a|)^2 (1+|a|s)(2-s) - \left( (2|a|-1)s + 2 - |a| \right) (1 - |a|s)^2 \leq 0.
\]
Define
\[
f(s) = (2|a|^2 - |a|) s^2 + (2|a|^2 - 5|a| + 3) s + 2|a| - 3.
\]
A simple computation shows that the previous inequality is equivalent to
\[
|a|(1-s) f(s) \leq 0.
\]
Since \(|a|(1-s) \geq 0\), we need to show \(f(s) \leq 0\) for \(s \in [0,1]\).

Differentiating \(f\) gives
\[
f'(s) = (4|a|^2 - 2|a|) s + (2|a|^2 - 5|a| + 3).
\]
We consider two cases:
\begin{itemize}
\item If \(0 \leq |a| < 1/2\), then \(4|a|^2 - 2|a| \leq 0\), and
\[
f'(s) \geq (4|a|^2 - 2|a|) + (2|a|^2 - 5|a| + 3) = 6|a|^2 - 7|a| + 3 > 0.
\]
\item If \(1/2 \leq |a| \leq 1\), then \(4|a|^2 - 2|a| \geq 0\), and
\[
f'(s) \geq 2|a|^2 - 5|a| + 3 \geq 0.
\]
\end{itemize}
In both cases, \(f'(s) \geq 0\) for \(s \in [0,1]\), so \(f\) is increasing. Therefore,
\[
f(s) \leq f(1) = 4|a|^2 - 4|a| \leq 0,
\]
which completes the proof of the upper bound.

For the lower bound, note that \(\phi^{-1}(x) = A^{-1}(\sigma^{-1}(x)) = A^{-1}(\sigma(x))\), and \(A^{-1}\) is orthogonal. Hence, for \(x, y \in \mathbb{B}^n\),
\[
\frac{\tilde{S}_{\mathbb{S}^{n-1},c}(x,y)}{\tilde{S}_{\mathbb{S}^{n-1},c}(\phi(x),\phi(y))} = \frac{\tilde{S}_{\mathbb{S}^{n-1},c}(\phi^{-1}(\phi(x)),\phi^{-1}(\phi(y)))}{\tilde{S}_{\mathbb{S}^{n-1},c}(\phi(x),\phi(y))} \leq \frac{1+|a|}{1-|a|},
\]
which implies
\[
\tilde{S}_{\mathbb{S}^{n-1},c}(\phi(x),\phi(y)) \geq \frac{1-|a|}{1+|a|} \tilde{S}_{\mathbb{S}^{n-1},c}(x,y).
\]
This establishes the lower bound and concludes the proof.
\end{proof}

\bibliographystyle{amsplain}
\bibliography{references}

@incollection{BeardonMinda2007,
    author = {Beardon, A. F. and Minda, D.},
    title = {The hyperbolic metric and geometric function theory},
    booktitle = {Quasiconformal Mappings and Their Applications},
    publisher = {Narosa},
    address = {New Delhi},
    year = {2007},
    pages = {9--56}
}

@book{KeenLakic2007,
    author = {Keen, L. and Lakic, N.},
    title = {Hyperbolic Geometry from a Local Viewpoint},
    series = {London Mathematical Society Student Texts},
    volume = {68},
    publisher = {Cambridge University Press},
    address = {Cambridge},
    year = {2007}
}

@article{GehringOsgood1979,
    author = {Gehring, F. W. and Osgood, B. G.},
    title = {Uniform domains and the quasihyperbolic metric},
    journal = {Journal d'Analyse Mathématique},
    volume = {36},
    pages = {50--74},
    year = {1979}
}

@article{GehringPalka1976,
    author = {Gehring, F. W. and Palka, B. P.},
    title = {Quasiconformally homogeneous domains},
    journal = {Journal d'Analyse Mathématique},
    volume = {30},
    pages = {172--190},
    year = {1976}
}

@incollection{Beardon1998,
    author = {Beardon, A. F.},
    title = {The {Apollonian} metric of a domain in $\mathbb{R}^n$},
    booktitle = {Quasiconformal Mappings and Analysis},
    editor = {Duren, P. and Heinonen, J. and Osgood, B. and Palka, B.},
    publisher = {Springer},
    address = {New York},
    year = {1998},
    pages = {91--108}
}

@article{Hasto2006,
    author = {Hästö, P.},
    title = {Gromov hyperbolicity of the $j_G$ and $\tilde{j}_G$ metrics},
    journal = {Proceedings of the American Mathematical Society},
    volume = {134},
    number = {4},
    pages = {1137--1142},
    year = {2006}
}

@article{HastoIbragimovLinden2006,
    author = {Hästö, P. and Ibragimov, Z. and Linden, H.},
    title = {Isometries of relative metrics},
    journal = {Computational Methods and Function Theory},
    volume = {6},
    number = {1},
    pages = {15--28},
    year = {2006}
}

@article{HastoLinden2004,
    author = {Hästö, P. and Linden, H.},
    title = {Isometries of the half-{Apollonian} metrics},
    journal = {Complex Variables, Theory and Application},
    volume = {49},
    number = {6},
    pages = {405--415},
    year = {2004}
}

@book{HaririKlenVuorinen2020,
    author = {Hariri, P. and Klén, R. and Vuorinen, M.},
    title = {Conformally Invariant Metrics and Quasiconformal Mappings},
    series = {Springer Monographs in Mathematics},
    publisher = {Springer},
    address = {Cham},
    year = {2020}
}

@article{HaririVuorinenZhang2017,
    author = {Hariri, P. and Vuorinen, M. and Zhang, X.},
    title = {Inequalities and bi-{L}ipschitz conditions for triangular ratio metric},
    journal = {Rocky Mountain Journal of Mathematics},
    volume = {47},
    number = {4},
    pages = {1121--1148},
    year = {2017}
}

@article{Ibragimov2016,
    author = {Ibragimov, Z.},
    title = {A scale-invariant {Cassinian} metric},
    journal = {Journal of Analysis},
    volume = {24},
    pages = {111--129},
    year = {2016}
}

@article{Ibragimov2019,
    author = {Ibragimov, Z.},
    title = {{Möbius} invariant {Cassinian} metric},
    journal = {Bulletin of the Malaysian Mathematical Sciences Society},
    volume = {42},
    pages = {1349--1367},
    year = {2019}
}

@article{Seittenranta1999,
    author = {Seittenranta, P.},
    title = {{Möbius}-invariant metrics},
    journal = {Mathematical Proceedings of the Cambridge Philosophical Society},
    volume = {125},
    number = {3},
    pages = {511--533},
    year = {1999}
}

@book{Vuorinen1988,
    author = {Vuorinen, M.},
    title = {Conformal Geometry and Quasiregular Mappings},
    series = {Lecture Notes in Mathematics},
    volume = {1319},
    publisher = {Springer},
    address = {Berlin},
    year = {1988}
}

@article{WangVuorinenZhang2021,
    author = {Wang, G. and Vuorinen, M. and Zhang, X.},
    title = {On cyclic quadrilaterals in {E}uclidean and hyperbolic geometries},
    journal = {Publicationes Mathematicae Debrecen},
    volume = {99},
    number = {1-2},
    pages = {123--140},
    year = {2021}
}

@article{WangXuVuorinen2021,
    author = {Wang, G. and Xu, X. and Vuorinen, M.},
    title = {Remarks on the scale-invariant {Cassinian} metric},
    journal = {Bulletin of the Malaysian Mathematical Sciences Society},
    volume = {44},
    pages = {1559--1577},
    year = {2021}
}

@article{XuWangZhang2022,
    author = {Xu, X. and Wang, G. and Zhang, X.},
    title = {Comparison and {Möbius} quasi-invariance properties of {Ibragimov}'s metric},
    journal = {Computational Methods and Function Theory},
    volume = {22},
    pages = {609--627},
    year = {2022}
}

@article{Hasto2002,
    author = {Hästö, P. A.},
    title = {A new weighted metric, the relative metric},
    journal = {Journal of Mathematical Analysis and Applications},
    volume = {274},
    number = {1},
    pages = {38--58},
    year = {2002}
}

@article{BonkKleiner2002,
    author = {Bonk, M. and Kleiner, B.},
    title = {Rigidity for quasi-{Möbius} group actions},
    journal = {Journal of Differential Geometry},
    volume = {61},
    number = {1},
    pages = {81--106},
    year = {2002}
}

@article{CuiXiao2022,
    author = {Cui, Y. M. and Xiao, Y. Q.},
    title = {A new intrinsic metric on metric spaces},
    journal = {Bulletin of the Malaysian Mathematical Sciences Society},
    volume = {45},
    pages = {2941--2958},
    year = {2022}
}

@article{ZhangXiao2019,
    author = {Zhang, Z. Q. and Xiao, Y. Q.},
    title = {Strongly hyperbolic metrics on {Ptolemy} spaces},
    journal = {Journal of Mathematical Analysis and Applications},
    volume = {478},
    number = {1},
    pages = {445--457},
    year = {2019}
}

@book{Beardon1995,
    author = {Beardon, A. F.},
    title = {The Geometry of Discrete Groups},
    series = {Graduate Texts in Mathematics},
    volume = {91},
    publisher = {Springer},
    address = {New York},
    year = {1995}
}

@book{GehringMartinPalka2017,
    author = {Gehring, F. W. and Martin, G. J. and Palka, B. P.},
    title = {An Introduction to the Theory of Higher-Dimensional Quasiconformal Mappings},
    series = {Mathematical Surveys and Monographs},
    volume = {216},
    publisher = {American Mathematical Society},
    address = {Providence, RI},
    year = {2017}
}

@book{Vaisala1971,
    author = {Väisälä, J.},
    title = {Lectures on $n$-Dimensional Quasiconformal Mappings},
    series = {Lecture Notes in Mathematics},
    volume = {229},
    publisher = {Springer},
    address = {Berlin},
    year = {1971}
}

@article{SimicVuorinenWang2015,
    author = {Simic, S. and Vuorinen, M. and Wang, G. D.},
    title = {Sharp {L}ipschitz constants for the distance ratio metric},
    journal = {Mathematica Scandinavica},
    volume = {116},
    number = {1},
    pages = {86--103},
    year = {2015}
}

\end{document}